\documentclass[12pt]{article} 
\usepackage{amsmath,amssymb,amsthm,color, graphicx}
\usepackage{bm}

\theoremstyle{definition}
\newtheorem{thm}{Theorem}[section]
\newtheorem{prop}[thm]{Proposition}     
\newtheorem{lem}[thm]{Lemma}

\newtheorem{rem}[thm]{Remark}

\numberwithin{equation}{section}

\newcommand{\R }{\mathbf{R}}
\newcommand{\E}{\mathbf{E}}

\newcommand{\cA}{{\cal A}}

\newcommand{\al}{\alpha}
\newcommand{\be}{\beta}
\newcommand{\gm}{\gamma}
\newcommand{\de}{\delta}
\newcommand{\ep}{\varepsilon}
\newcommand{\lm}{\lambda}
\newcommand{\kp}{\kappa}
\newcommand{\De}{\Delta}

\newcommand{\supp}{\operatorname{supp}}
\newcommand{\Int}{\operatorname{int}}

\title{Sharp estimates of the generalized principal eigenvalue for superlinear viscous Hamilton-Jacobi equations with inward drift}

\author{Emmanuel Chasseigne\footnote{Institut Denis Poisson (UMR CNRS 7013),
        Universit\'e de Tours, Parc de Grandmont, 37200 Tours, France. Email: {\tt emmanuel.chasseigne@lmpt.univ-tours.fr}.}  \ and  Naoyuki Ichihara\footnote{Department of Physics and Mathematics, Aoyama Gakuin University, 5-10-1 Fuchinobe, Chuo-ku, Sagamihara-shi, 
Kanagawa 252-5258, Japan. Email: {\tt ichihara@gem.aoyama.ac.jp}. }}

\begin{document}
\maketitle

\begin{abstract}
This paper is concerned with the ergodic problem for viscous Hamilton-Jacobi equations having superlinear Hamiltonian, inward-pointing drift, and positive potential which vanishes at infinity. Assuming some radial symmetry of the drift and the potential outside a ball, we establish sharp estimates of the generalized principal eigenvalue with respect to a perturbation of the potential.
\end{abstract}

\section{Introduction and Main results}

In this paper we complete the analysis we performed in the previous paper
\cite{CI18} by providing new results as well as sharp estimates for the
generalized principal eigenvalue of superlinear viscous Hamilton-Jacobi
equations. More specifically, we consider the following equation with
superlinear exponent $m>1$ and real parameter $\be\geq 0$:
\begin{equation}\tag{EP}
\lm-\De\phi(x)+b(x)\cdot D\phi(x)+\frac{1}{m}|D\phi(x)|^m-\be V(x)=0,\qquad
x\in\R^d\,.
\end{equation}
Here, $D\phi$ and $\De\phi$ denote, respectively, the gradient and the Laplacian
of $\phi=\phi(x)$. The unknown of (EP) is the pair of a real constant $\lm$ and
a function $\phi:\R^d\to\R$. Finding such a pair is called the ergodic problem.
Throughout the paper, we assume that drift vector field $b:\R^d\to \R^d$ is
inward-pointing outside $B_R:=\{x\in\R^d\,|\, |x|<R\}$ for some $R>0$, and that
potential function $V:\R^d\to\R$ is nonnegative, vanishing at infinity.
A typical example of such $(b,V)$ is 
\begin{equation}\label{bV1}
b(x)=\rho|x|^{\de-1}x+a|x|^{\de-2}x,\qquad 
V(x)=|x|^{-\eta},\qquad 
\forall x\not\in B_R,
\end{equation}
where $\rho>0$, $\de\in\R$, $a\in\R$, $\eta>0$, and $R>0$ are some constants,
but we also obtain some results when $V$ is compactly supported.
The precise assumption on $(b,V)$ will be stated in the next section.

The main purpose of this paper is to investigate the asymptotic behavior, as $\be\to\infty$, of the generalized principal eigenvalue of (EP) defined by
\begin{equation}\label{lm_max}
\lm_{\max}(\be):=\sup\{\lm\in\R\,|\, \text{(EP) has a subsolution $(\lm,\phi)\in \R\times C^3(\R^d)$}\},
\end{equation}
where $(\lm,\phi)$ is said to be a subsolution (resp. solution) of (EP) if the left-hand side of (EP) is non-positive (resp. zero) for every $x\in\R^d$. In what follows, we regard $\lm_{\max}(\be)$ as a function of $\be\in[0,\infty)$ and call it {\em the spectral function} of (EP). 
We mention that, if $m=2$ in (EP), then our ergodic problem (EP) is equivalent to the following linear eigenvalue problem:
\begin{equation}\label{linear}
Lu=\lm u\quad \text{in }\ \R^d,\qquad L:=-2\De+2b\cdot D+\be V.
\end{equation}
Indeed, by setting $u:=e^{-(1/2)\phi}$, one can verify that $(\lm,\phi)$ is a solution of (EP) if and only if $(\lm,u)$ satisfies (\ref{linear}). Furthermore, $\lm_{\max}(\be)$ defined above coincides with the generalized principal eigenvalue of the linear differential operator $L$ in (\ref{linear}) (c.f., \cite[Section 4.3]{Pi95}). We notice that there is no such transformation for $m\ne 2$, so that (EP) should be treated as a nonlinear (additive) eigenvalue problem.

There is an extensive literature on the shape of the spectral function for many types of linear differential operators under various settings. In connection with the criticality theory, we refer to \cite{Da95,Mu86,Si81,Si83,Si84a,Si84b,Si85} in which is studied the principal eigenvalue for Schr\"odinger operators $L=-\De +\be V$, and to \cite{ABS17,Pi89a,Pi89b,Pi90,Pi95,Si81} for more general second order elliptic operators. It is worth mentioning that \cite{Ta14,Ta16,TT04,TT07,TU04} investigate the criticality of differential operators of the form $L=(-\De)^{\frac{\al}{2}}+\be \mu $, where $0<\al\leq 2$ and $\mu$ is a suitable Radon measure on $\R^d$. 

As mentioned, ergodic problem (EP) is a nonlinear extension of linear eigenvalue problem (\ref{linear}). From this point of view, a similar criticality theory is discussed in \cite{CI16,CI18,Ic13,Ic15} for viscous Hamilton-Jacobi equations of the form (EP) (see also \cite{Ic16} for a discrete analogue of it). The present paper is closely related to \cite{CI18} where the asymptotic behavior of $\lm_{\max}(\be)$ as $\be\to\infty$ is investigated under some conditions including (\ref{bV1}); it is proved that the spectral function is non-decreasing and concave in $[0,\infty)$ with $\lm_{\max}(0)=0$, and that the limit of $\lm_{\max}(\be)$ as $\be\to\infty$ is determined according to the strength of the inward-pointing drift $b$. More precisely, the following holds:
\begin{equation}\label{lm_bar}
\overline{\lm}_{\max}:=\lim_{\be\to\infty}\lm_{\max}(\be)=
\begin{cases}
+\infty &\text{if } \de>0,\\
\text{positive and finite} &\text{if } \de=0,\\
0 &\text{if } \de<0.
\end{cases}
\end{equation}
In order to distinguish these three types of drifts, we say that $b$ is a {\it strong} drift if $\de>0$, a {\it moderate} drift if $\de=0$, and a {\it weak} drift if $\de<0$. 

The goal of this paper is to give a complete picture of the above asymptotic problem by establishing sharp estimates of (\ref{lm_bar}). 
Our main results consist of two parts. The first result concerns the strong drift case (i.e., $\de>0$); we show under condition (\ref{bV1}) that
\begin{equation}\label{c_0}
\lim_{\be\to\infty}\be^{-\frac{\de m^\ast }{\de m^\ast +\eta}}\lm_{\max}(\be)=c_0,
\end{equation}
where $m^\ast:=m/(m-1)$, and $c_0>0$ is some constant depending on $\de$,
$\rho$, $m$, and $\eta$ (but not on $a$). See Theorem \ref{main1} of Section 3
for the explicit form of $c_0$. Note that a rough estimate
$\lm_{\max}(\be)=O(\be^{\frac{\de m^\ast }{\de m^\ast +\eta}})$ has been
obtained in \cite[Theorem 2.3]{CI18}, but the new ingredient here is that we show
existence of a limit of $\be^{-\frac{\de m^\ast }{\de m^\ast
        +\eta}}\lm_{\max}(\be)$ as $\be\to\infty$, together with its explicit
form. To our best of knowledge, (\ref{c_0}) has not been obtained even for the linear operator (\ref{linear}), namely, the case where $m=2$ in our context.

The second main result is concerned with the moderate drift case  (i.e.,
$\de=0$). In view of (\ref{lm_bar}), one of the following (a) or (b) happens:\\
 \ \ (a) there exists a $\be_0\in(0,\infty)$ such that $\lm_{\max}(\be)=\overline{\lm}_{\max}$ for all $\be\geq \be_0$; \\
 \ \ (b) $\lm_{\max}(\be)<\overline{\lm}_{\max}$ for all $\be\geq 0$. \\
In \cite{CI18}, it is proved that (a) occurs provided $(b,V)$ satisfies
condition (\ref{bV1}) with $a=0$ and $0<\eta\leq 1$. In this situation, 
we face a \textit{plateau} since $\lm_{\max}(\beta)$ does not increase anymore after $\beta_0$.
The novelty of this paper is that we clarify the role of parameter $a$ appearing in (\ref{bV1}); it turns out that, contrary to the strong drift case, constant $a$ plays a crucial role
in determining the asymptotic behavior of $\lm_{\max}(\be)$. In fact, we prove in Theorem \ref{main2} of Section 4 that, under condition (\ref{bV1}) with $\eta>1$, situation (a) occurs if and only if $a\geq d-1$, where $d$ denotes the
space dimension. This allows one to characterize the shape of the spectral function in terms of $(a,\eta)$. Furthermore, in the case where situation (b) occurs (i.e., for $a<d-1$ and $\eta>1$), we show that
\begin{equation}\label{c_1}
\lim_{\be\to\infty}\be^{\frac{1}{\eta-1}}(\overline{\lm}_{\max}-\lm_{\max}(\be))=c_1
\end{equation}
for some $c_1>0$ having an explicit form in terms of $m,\rho,\eta,a$ and $d$. Such a characterization has not been studied in the existing literature. 

We finish this introduction with a stochastic control interpretation of $\lm_{\max}(\be)$. Let $\xi=(\xi_t)_{t\geq 0}$ be a control process with values in $\R^d$ belonging to a suitable admissible class $\cA$. For a given $\xi\in \cA$, we denote by $X^\xi=(X^\xi_t)_{t\geq 0}$ the associated controlled process governed by the following stochastic differential equation  in $\R^d$:
\begin{equation*}
dX^\xi_t=-\xi_t\,dt-b(X^\xi_t)\,dt+\sqrt{2}\,dW_t,\qquad t\geq 0,
\end{equation*}
where $W=(W_t)_{t\geq 0}$ stands for a $d$-dimensional standard Brownian motion defined on some probability space. The criterion to be minimized is 
\begin{equation*}
J_\be(\xi;x):=\limsup_{T\to\infty}\frac1T\E\left[\int_0^T\Big(\frac{1}{m^\ast}|\xi_t|^{m^\ast}+\be V(X^\xi_t)\Big)\,dt\,\Big|\,X_0=x\right],\quad \xi\in\cA,
\end{equation*}
where $m^\ast:=m/(m-1)$. 
Then, one can verify that $\inf_{\xi\in\cA}J_\be(\xi;x)=\lm_{\max}(\be)$ for all $x\in\R^d$ (see \cite{CI18, Ic15} for a rigorous justification of it). 
If $\beta$ is sufficiently large, then the optimal strategy for the controller is to avoid the region where $V$ is large. 
Since the drift vector field $b$ is inward-pointing, the controller is obliged to compensate it by choosing an appropriate outward-pointing vector $\xi_t$ at each time $t$. 
In this trade-off situation, the value of $\lm_{\max}(\be)$ is determined as the best balance between the cost incurred by $V$ and that by $b$. In this paper, we do not develop this probabilistic viewpoint and focus on its analytic counterpart.

This paper is organized as follows. The next section is devoted to some preliminaries. In Section 3, we investigate the asymptotic behavior of $\lm_{\max}(\be)$ in the strong drift case ($\de>0$) and establish sharp estimate (\ref{c_0}). Section 4 is concerned with the moderate drift case ($\de=0$). The shape of the spectral function is characterized in terms of $(a,\eta)$. In Section 5, we prove convergence (\ref{c_1}) when (b) happens in the moderate drift case. 

\section{Preliminaries}

Throughout the paper, we set $m^\ast:=m/(m-1)>1$, which is the conjugate number of $m$ satisfying $(1/m)+(1/m^\ast)=1$.
We assume that drift $b$ and potential $V$ in (EP) satisfy the following:
\begin{description}
\item[(H0)] $b\in C^{2}(\R^d,\R^d)$, $V\in C^2(\R^d)$, $DV\in L^\infty(\R^d,\R^d)$, and there exist some $\de>0$, $\rho>0$, $R>0$, and $\al \in C^2(\R^d)$ such that 
\begin{align*}
&b(x)=(\rho+\al(x))|x|^{\de-1}x\quad\text{for all }x\not\in B_R,\qquad \lim_{|x|\to\infty}\al(x)=0,\\
&V(x)\geq 0\quad\text{for all }x\in \R^d,\qquad \overline{B}_R\subset \Int(\supp V),\qquad \lim_{|x|\to\infty}V(x)=0,
\end{align*}
where $\overline{B}_R$ and $\Int(\supp V)$ stand for the closure of $B_R$ and the interior of the support of $V$, respectively.
\end{description}
Roughly speaking, condition (H0) requires that the intensity of the
inward-pointing drift is of order $|b(x)|=\rho|x|^{\de}+o(|x|^{\de})$ as
$|x|\to\infty$ for some $\rho>0$ and $\de>0$, and that $V$ is strictly positive
on the closed ball $\overline{B}_R$, in which region $b$ is not necessarily
inward-pointing. Notice that hypothesis (H0) can be satisfied by compactly
    supported potentials $V$ provided their support contains $B_{R'}$ for some
    $R'>R$.

We also consider the following stronger assumption (where here, $V$ cannot be
    compactly supported of course):
\begin{description}
\item[(H1)] $(b,V)$ satisfies (H0) and for some $a\in\R$ and $\eta>0$, 
\begin{equation*}
\al(x)=a|x|^{-1}, \quad V(x)=|x|^{-\eta}\quad \text{for all }x\not\in B_R
\end{equation*}
where $R$ is the constant in (H0).
\end{description}

We begin with recalling the solvability of (EP). 

\begin{thm}\label{pre1}
Let (H0) hold. Then, for any $\be\geq 0$, $\lm_{\max}(\be)$ defined by (\ref{lm_max}) is finite, and there exists a function $\phi\in C^3(\R^d)$ which satisfies (EP) with $\lm=\lm_{\max}(\be)$. Moreover, the following gradient estimate holds:
\begin{equation*}
|D\phi(x)|\leq K(1+|x|^{\frac{\de}{m-1}}+\max\{-\lm_{\max}(\be),0\}^{\frac1m}),\quad x\in\R^d,
\end{equation*}
where $K>0$ is a constant not depending on $\phi$.
\end{thm}
\begin{proof}
We refer, for instance, to \cite[Propositions 4.1, 4.2]{CI18} for a complete proof (see also \cite[Theorem 2.1]{Ic09} or \cite[Theorem 3.5]{Ci14}). Note that \cite{CI18} deals with the case where $\de\leq 1$ only, but their proof can also be applied to any $\de>1$ with no change. 
\end{proof}

The following theorem gives some rough estimate of $\lm_{\max}(\be)$ as $\be\to \infty$.
\begin{thm}\label{pre2}
Let (H0) hold. Then $\lm_{\max}(0)=0$, and the spectral function $\be \mapsto \lm_{\max}(\be)$ is non-constant, non-decreasing, and concave  in $[0,\infty)$. In particular, it exists a limit $\displaystyle \overline{\lm}_{\max}:=\lim_{\be\to\infty}\lm_{\max}(\be)\in (0,+\infty]$. 
Moreover, if (H1) holds, then $\lm_{\max}(\be)=O(\be^{\frac{\de m^\ast }{\de m^\ast +\eta}}) $ as $\be\to\infty$. 
\end{thm}

\begin{proof}
See \cite{CI18,Ic15} for details.
\end{proof}

\begin{rem}
If $\de<0$ in (H0), then $\lm_{\max}(\be)=0$ for all $\be\geq 0$ (see \cite[Proposition 6.2]{Ic15}). We excluded this trivial case from our assumption (H0).   
\end{rem}

\section{The sharp estimate for $\de>0$}
This section is devoted to the strong drift case ($\de>0$). 
In the rest of this paper, for $\phi\in C^2(\R^d)$, we set
\begin{equation*}
F[\phi](x):=-\De \phi(x)+b(x)\cdot D\phi(x)+\frac1m|D\phi(x)|^m,\quad x\in \R^d.
\end{equation*}
The main result of this section is the following. 
\begin{thm}\label{main1}
Let (H1) hold with $\de>0$. Then,
\begin{equation}\label{sharp1}
\lim_{\be\to\infty}\be^{-\frac{\de m^\ast }{\de m^\ast +\eta}}\,\lm_{\max}(\be)=\left(\frac{\de m^\ast +\eta}{\de m^\ast }\right)\left(\frac{\rho^{m^\ast}\de}{\eta}\right)^{\frac{\eta}{\de m^\ast +\eta}}=:c_0.
\end{equation}
\end{thm}

The proof is divided into two parts; we derive lower and upper bounds separately. 
We first consider the lower bound. 

\begin{prop}\label{prop.lower}
Let (H1) hold with $\de>0$. Then,
\begin{equation}\label{lower1}
\liminf_{\be\to\infty}\be^{-\frac{\de m^\ast }{\de m^\ast +\eta}}\,\lm_{\max}(\be)\geq c_0.
\end{equation}
\end{prop}

\begin{proof}
Let $\phi_0\in C^3(\R^d)$ be any function such that $\phi_0(x)=-\rho^{m^\ast-1}(1+\frac{\de}{m-1})^{-1}|x|^{1+\frac{\de}{m-1}}$ for all $x\not\in B_R$. Then, by direct computations, we observe that, for any $x\not\in B_R$, 
\begin{equation*}
F[\phi_0](x)
=\rho^{m^\ast-1}\Big(d+\frac{\de}{m-1}-1\Big)|x|^{\frac{\de}{m-1}-1}-\frac{\rho^{m^\ast}}{m^\ast}|x|^{\de m^\ast}-\rho^{m^\ast-1}a|x|^{\de m^\ast-1}.
\end{equation*}
Since $\de/(m-1)<\de m^\ast$, we have
\begin{align*}
F[\phi_0](x)-\be V(x)
&\leq C|x|^{\de m^\ast-1}-\frac{\rho^{m^\ast}}{m^\ast}|x|^{\de m^\ast }-\be|x|^{-\eta},\qquad x\not\in B_R,
\end{align*}
where $C>0$ is a constant not depending on $R$ and $\be$.

Set $f(r):=(\rho^{m^\ast}/m^\ast)r^{\de m^\ast }+\be r^{-\eta}$ for $r\geq R$. Note that $f$ attains its minimum at $r=\max\{R,r_0\}$, where
\begin{equation*}
 r_0:=\left(\frac{\eta\be}{\rho^{m^\ast}\de}\right)^{\frac{1}{\de m^\ast +\eta}}.
\end{equation*}
Thus, if we take $\be$ so large that $\be> \rho^{m^\ast}\de R^{\de m^\ast+\eta}/\eta$, then 
\begin{align*}
\min_{r\geq R}f(r)=f(r_0)&=\frac{\rho^{m^\ast}}{m^\ast}\Big(\frac{\eta\be}{\rho^{m^\ast}\de}\Big)^{\frac{\de m^\ast }{\de m^\ast +\eta}}+\be\Big(\frac{\eta\be}{\rho^{m^\ast}\de}\Big)^{-\frac{\eta}{\de m^\ast +\eta}}=c_0\be^{\frac{\de m^\ast }{\de m^\ast +\eta}}.
\end{align*}

Now, we set $f_\ep(r):=(\rho^{m^\ast}/m^\ast-\ep)r^{\de m^\ast }+\be r^{-\eta}$ for $r\geq R$ and $\ep>0$. Then, similarly as above, one can verify that, for each $\ep>0$ and $\be$ sufficiently large, $f_\ep $ attains its minimum at some $r=r_\ep$, that the minimum of $f_\ep$ has the form $f_\ep(r_\ep)=c_\ep\be^{\frac{\de m^\ast }{\de m^\ast +\eta}}$ for some $c_\ep>0$ not depending on $\be$, and that $r_\ep\to r_0$ and $c_\ep\to c_0$ as $\ep\to0$.
Taking these into account, we observe that, for any $\be$ sufficiently large,  
\begin{align*}
F[\phi_0](x)-\be V(x)
\leq C|x|^{\de m^\ast-1}-\ep|x|^{\de m^\ast }-f(r_\ep),\quad x\not\in B_R.
\end{align*}
Choosing $R_\ep\geq R$ so large that $C|x|^{\de m^\ast-1}\leq \ep|x|^{\de m^\ast }$ for all $x\not\in B_{ R_\ep}$, we obtain
\begin{align*}
F[\phi_0](x)-\be V(x)
\leq -f(r_\ep),\quad x\not\in B_{R_\ep}.
\end{align*}
Furthermore, since $\de m^\ast/(\de m^\ast+\eta)<1$, one can find a $\be_0$ such that, for any $\be\geq \be_0$.
\begin{equation*}
\be>\frac{f(r_\ep)+\sup_{B_{R_\ep}}F[\phi_0]}{\inf_{B_{R_\ep}}V}.
\end{equation*}
This implies that
\begin{align*}
F[\phi_0](x)-\be V(x)
\leq -f(r_\ep),\quad x\in B_{R_\ep}.
\end{align*}
Thus, the pair $(f(r_\ep),\phi_0)$ is a subsolution of (EP). By the definition of $\lm_{\max}(\be)$, we conclude that $\lm_{\max}(\be)\geq f(r_\ep)=c_\ep\be^{\frac{\de m^\ast }{\de m^\ast +\eta}}$ for any $\be\geq \be_0$. In particular,
\begin{equation*}
\liminf_{\be\to\infty}\be^{-\frac{\de m^\ast }{\de m^\ast +\eta}}\lm_{\max}(\be)\geq c_\ep.
\end{equation*}
Letting $\ep\to0$ and noting that $c_\ep\to c_0$ as $\ep\to0$, we obtain (\ref{lower1}). 
\end{proof}

We next give the upper bound with the same constant $c_0$.  

\begin{prop}\label{de>0_upper}
Let (H1) hold with $\de>0$. Then,
\begin{equation}\label{upper1}
\limsup_{\be\to\infty}\be^{-\frac{\de m^\ast }{\de m^\ast +\eta}}\,\lm_{\max}(\be)\leq c_0.
\end{equation}
\end{prop}

\begin{proof}
Let $(\lm_{\max}(\be),\phi)\in \R\times C^3(\R^d)$ be a  solution of (EP). Then
\begin{equation}\label{phi_sub}
\lm_{\max}(\be)-\De\phi(x)+b(x)\cdot D\phi(x)+\frac1m|D\phi(x)|^m-\be V(x)= 0
\end{equation}
for all $x\in\R^d$.
Let $\zeta\in C_0^\infty(\R^d)$ be a test function such that $\zeta\geq 0$ in $\R^d$, $\int_{\R^d}\zeta^{m^\ast}dx=1$, and $\supp \zeta\subset B_1$. 
We choose an arbitrary $y\not\in B_{ R+1}$ and set $\zeta_y:=\zeta(\,\cdot\,-y)$.  Then, multiplying both sides of (\ref{phi_sub}) by $\zeta_y^{m^\ast}$, using the integration by parts formula, and noting the fact that $(\supp \zeta_y)\cap B_R=\emptyset$, we have
\begin{align*}
\lm_{\max}(\be)
&= -m^\ast\int_{\R^d}\zeta_y^{m^\ast-1}D\zeta_y\cdot D\phi\, dx
 -\int_{\R^d}\zeta_y^{m^\ast}(b\cdot D\phi+\frac1m|D\phi|^{m}-\be V)\, dx\\
 &= -m^\ast\int_{\R^d}D\zeta_y\cdot (\zeta_y^{m^\ast-1}D\phi)\, dx-\frac1m\int_{\R^d}\zeta_y^{m^\ast}|D\phi|^{m}\,dx\\
&\qquad -\int_{\R^d}\zeta_y^{m^\ast}(\rho|x|^{\de-1}x\cdot D\phi+a|x|^{\de-2}x\cdot D\phi-\be |x|^{-\eta})\, dx.
\end{align*}
We now recall Young's inequality of the form 
\begin{equation*}
p\cdot q\leq \frac1m|cp|^m+\frac{1}{m^\ast}|c^{-1}q|^{m^\ast},\qquad p,q\in\R^d,\quad c>0,
\end{equation*}
and apply it to $-(\rho|x|^{\de-1}x)\cdot D\phi$.
Then, for any $l>m$, there exists some $C_1(l)>0$ with $C_1(l)\to 1/m^\ast$ as $l\to m$ such that

\begin{align*}
\lm_{\max}(\be)
&\leq  -m^\ast\int_{\R^d}D\zeta_y\cdot (\zeta_y^{m^\ast-1}D\phi)\, dx-\Big(\frac1m-\frac1l\Big)\int_{\R^d}\zeta_y^{m^\ast}|D\phi|^{m}\,dx\\
&\qquad +\int_{\R^d}\zeta_y^{m^\ast}(C_1(l)\rho^{m^\ast}|x|^{\de m^\ast})\,dx
-\int_{\R^d}\zeta_y^{m^\ast}(a|x|^{\de-2}x\cdot D\phi-\be |x|^{-\eta})\, dx.
\end{align*}

Furthermore, applying Young's inequality to $D\zeta_y\cdot (\zeta_y^{m^\ast-1}D\phi)$ and $(a|x|^{\de-2}x)\cdot D\phi$, we can see that, for the above $l>m$, there exists another constant $C_2(l)>0$ with $C_2(l)\to \infty$ as $l\to m$ such that
\begin{align*}
\lm_{\max}(\be)
&\leq C_2(l)\int_{\R^d}(|D\zeta_y|^{m^\ast}+\zeta_y^{m^\ast}|a|^{m^\ast}|x|^{(\de-1) m^\ast})dx\\
&\qquad +\int_{\R^d}\zeta_y^{m^\ast}(C_1(l)\rho^{m^\ast}|x|^{\de m^\ast}+\be |x|^{-\eta})\,dx\\
&= C_2(l)\int_{\R^d}(|D\zeta(x)|^{m^\ast}+\zeta(x)^{m^\ast}|a|^{m^\ast}|x+y|^{(\de-1) m^\ast})dx\\
&\qquad +\int_{\R^d}\zeta(x)^{m^\ast}(C_1(l)\rho^{m^\ast}|x+y|^{\de m^\ast}+\be |x+y|^{-\eta})\,dx.
\end{align*}
Since $|x|/|y|< 1/R$ and $|x+y|>R$ for any  $x\in \supp \zeta$, we observe that, for any $\gm\in\R$ and $x\in\supp\zeta$, 
\begin{equation*}
|x+y|^\gm\leq |y|^\gm\Big(1+\frac{|x|}{|y|}\Big)^\gm
\leq |y|^\gm\Big(1+C_0\frac{|x|}{|y|}\Big)=|y|^\gm+C_0|x|\,|y|^{\gm-1},
\end{equation*}
where $C_0>0$ is a constant depending only on $R$ and $\gm$.
Noting this and the fact that $\int_{\R^d}|x|\zeta(x)^{m^\ast}dx\leq \int_{\R^d}\zeta(x)^{m^\ast}dx=1$, we have 
\begin{align*}
\lm_{\max}(\be)
&\leq C_2(l)\int_{\R^d}|D\zeta|^{m^\ast}dx+C_2(l)|a|^{m^\ast}|y|^{(\de-1)m^\ast}+C_1(l)\rho^{m^\ast}|y|^{\de m^\ast }+\be |y|^{-\eta}\\
&\quad+C(C_2(l)|a|^{m^\ast}|y|^{(\de-1)m^\ast-1}+C_1(l)\rho^{m^\ast}|y|^{\de m^\ast -1}+\be|y|^{-\eta-1})
\end{align*}
for some $C>0$. In what follows, $C>0$ denotes various constants not depending on $\be$. 

We now set $|y|=\theta\be^{\frac{1}{\de m^\ast +\eta}}$ for a large $\be$, where $\theta>0$ will be optimized later. Then, plugging this into the above inequality, we obtain an estimate of the form
\begin{align*}
\lm_{\max}(\be)&\leq C(1+\be^{\frac{\de m^\ast -1}{\de m^\ast +\eta}})+ (C_1(l)\rho^{m^\ast}\theta^{\de m^\ast }+\theta^{-\eta})\be^{\frac{\de m^\ast }{\de m^\ast +\eta}}.
\end{align*}
Multiplying both sides by $\be^{-\frac{\de m^\ast }{\de m^\ast +\eta}}$ and letting $\be\to\infty$, we have
\begin{align*}
\limsup_{\be\to\infty}\be^{-\frac{\de m^\ast }{\de m^\ast +\eta}}\,\lm_{\max}(\be)\leq C_1(l)\rho^{m^\ast}\theta^{\de m^\ast }+\theta^{-\eta}.
\end{align*}
Sending $l\to m$ and noting that $C_1(l)\to1/m^\ast$ as $l\to m$, we conclude that
\begin{align*}
\limsup_{\be\to\infty}\be^{-\frac{\de m^\ast }{\de m^\ast +\eta}}\,\lm_{\max}(\be)&\leq \frac{\rho^{m^\ast}}{m^\ast}\,\theta^{\de m^\ast }+\theta^{-\eta}=:g(\theta).
\end{align*}
Since $g(\theta)$ attains its global minimum at $\theta_0:=(\de\rho^{m^\ast}/\eta)^{-\frac{1}{\de m^\ast +\eta}}$ and 
\begin{equation*}
g(\theta_0)=\frac{\rho^{m^\ast}}{m^\ast}\Big(\frac{\rho^{m^\ast}\de}{\eta}\Big)^{-\frac{\de m^\ast}{\de m^\ast+\eta}}+\Big(\frac{\rho^{m^\ast}\de}{\eta}\Big)^{\frac{\eta}{\de m^\ast+\eta}}=c_0,
\end{equation*}
we obtain the desired estimate (\ref{upper1}).
\end{proof}

Theorem \ref{main1} is a direct consequence of the above two propositions. 
We notice that the strict positivity assumption of $V$ in (H1) is crucial to the
proof of Theorem \ref{main1}. However, by a careful reading of the proof of
Proposition \ref{de>0_upper}, one can verify the following theorem under the
less restrictive hypothesis (H0):
\begin{thm}
Let (H0) hold with $\de>0$.  Assume that $\al(x)=a|x|^{-1}$ outside $B_R$ for some $a\in\R$ and $\supp V\subset B_{R'}$ for some $R'> R$. Then, $\displaystyle \overline{\lm}_{\max}:=\lim_{\be\to\infty}\lm_{\max}(\be)<\infty$. 
\end{thm}
Indeed, in order to get this upper bound it is enough that $V(x)\leq
    c|x|^{-\eta}$ outside $B_R$, which is of course the case if $V$ is compactly
    supported and satisfies (H0).
  
\section{The shape of $\lm_{\max}(\be)$ for $\de=0$} 
This section is concerned with the moderate drift case ($\de=0$). We first give an upper bound of $\lm_{\max}(\be)$. 
\begin{prop}\label{de=0}
Let (H0) hold with $\de=0$. Then $\lm_{\max}(\be)\leq \rho^{m^\ast}/m^\ast$ for any $\be\geq 0$.
\end{prop}

\begin{proof}
Let $\zeta\in C^\infty(\R^d)$ be a test function such that $\zeta\geq 0$ in $\R^d$,  $\int_{\R^d}\zeta^{m^\ast}dx=1$, and $(\supp\zeta)\cap B_R=\emptyset$. Let $(\lm_{\max}(\be),\phi)\in \R\times C^3(\R^d)$ be a solution of (EP). Then, similarly as in the proof of Proposition \ref{upper1}, we see that, for any $l>m$, there exist some $C_1(l),C_2>0$ with $C_1(l)\to 1/m^\ast$ and $C_2(l)\to \infty $ as $l\to m$ such that
\begin{align*}
\lm_{\max}(\be)\leq C_2(l)\int_{\R^d}( |D\zeta|^{m^\ast}+\zeta^{m^\ast}|\al|^{m^\ast})dx
+\int_{\R^d} \zeta^{m^\ast}(C_1(l)\rho^{m^\ast}+\be V)\,dx.
\end{align*}

We fix an arbitrary $\ep\in(0,1)$ and set $\zeta_\ep(x):=\ep^{\frac{d}{m^\ast}}\zeta(\ep x)$ for $x\in\R^d$. Then,
\begin{equation*}
\int_{\R^d}(\zeta_\ep)^{m^\ast}\,dx=1,\quad 
(\supp \zeta_\ep)\cap B_{\ep^{-1}R}=\emptyset,\quad 
\int_{\R^d}|D\zeta_\ep|^{m^\ast}\,dx\leq C\ep^{m^\ast}
\end{equation*}
for some $C>0$ not depending on $\ep$. In what follows $C$ denotes various constants not depending on $\ep$. We also fix an arbitrary $\theta>0$ and choose an $\ep_0>0$ so that $|\al(x)|+V(x)<\theta$ for all $x\not\in B_{\ep_0^{-1}R}$. 
Then, plugging $\zeta_\ep$, with $\ep<\ep_0$, into the above $\zeta$, we have
\begin{align*}
\lm_{\max}(\be)&\leq C_2(l)C\ep^{m^\ast}+C_2(l)\int_{B_{\ep_0^{-1}R}}\zeta_\ep^{m^\ast}|\al|^{m^\ast}dx+C_2(l)\theta^{m^\ast}\\
&\qquad +C_1(l)\rho^{m^\ast}+\be \int_{B_{\ep_0^{-1}R}}\zeta_\ep^{m^\ast}V\,dx+\be \theta.
\end{align*}
Since $B_{\ep_0^{-1}R}$ is bounded and $\zeta_\ep\to 0$ in $\R^d$ as $\ep\to 0$, we see by sending $\ep\to0$ in the above inequality that
\begin{equation*}
\lm_{\max}(\be)\leq C_2(l)\theta^{m^\ast}+C_1(l)\rho^{m^\ast}+\be \theta.
\end{equation*}
Letting $\theta\to 0$ and $l\to m$, we conclude that $\lm_{\max}(\be)\leq \rho^{m^\ast}/m^\ast$. Since $\be$ is arbitrary, we obtain the desired estimate. 
\end{proof}

The next proposition gives a lower bound under some structure condition on $(\al,V)$. 
\begin{prop}\label{de=0_lower}
Let (H0) hold with $\de=0$. Assume that there exist some $k_1>0$ and $R_1\geq R$ such that
\begin{equation}\label{structure1}
|x|(\al(x)+k_1V(x))\geq d-1,\qquad \forall x\not\in B_{R_1}.
\end{equation}
Then, there exists a $\be_0\geq 0$ such that $\lm_{\max}(\be)=\rho^{m^\ast}/m^\ast$ for all $\be\geq \be_0$.
\end{prop}

\begin{proof}
Let $\phi_0\in C^3(\R^d)$ be any function such that
$\phi_0(x)=-\rho^{m^\ast-1}|x|$ for all $x\not\in B_{R}$. Then we observe by direct computations that, for any $x\not\in B_R$,
\begin{align*}
F[\phi_0](x) &=\frac{(d-1)\rho^{m^\ast-1}}{|x|}-\rho^{m^\ast-1}(\rho+\al(x))+
 \frac{\rho^{m^\ast}}{m}\\
 &=\frac{(d-1)\rho^{m^\ast-1}}{|x|}-\rho^{m^\ast-1}\al(x)-
 \frac{\rho^{m^\ast}}{m^\ast}.
\end{align*}
In particular, for any $x\not\in B_R$, 
\begin{equation*}
\frac{\rho^{m^\ast}}{m^\ast}+F[\phi_0](x)-\beta V(x)\leq \frac{\rho^{m^\ast-1}}{|x|}\big\{d-1-|x|(\al(x)+\beta \rho^{1-m^\ast}V(x))\big\}.
\end{equation*}
Setting $\be_1:=\rho^{m^\ast-1}k_1$ and choosing  $\be\geq \be_1$, we see in view of (\ref{structure1}) that
\begin{equation*}
\frac{\rho^{m^\ast}}{m^\ast}+F[\phi_0](x)-\beta V(x)\leq 0,
\quad  x\not\in B_{R_1}.
\end{equation*}
Taking $\be$ so large that $\be\geq \beta_2:=(\inf_{B_{R_1}}V)^{-1}(\rho^{m^\ast}/m^\ast+\sup_{B_{R_1}}F[\phi_0])$, we also obtain
\begin{equation*}
\frac{\rho^{m^\ast}}{m^\ast}+F[\phi_0](x)-\beta V(x)\leq 0,
\quad  x\in B_{R_1}.
\end{equation*}
Thus, $(\rho^{m^\ast}/m^\ast,\phi_0)$ is a subsolution of (EP) for any $\beta\geq \be_0:=\max\{\beta_1,\beta_2\}$. This implies by the definition of $\lm_{\max}(\be)$ that $\lm_{\max}(\be)\geq \rho^{m^\ast}/m^\ast$ for all $\beta\geq \be_0$.
\end{proof}

Now, we seek for a sufficient condition so that the opposite situation happens. To this end, we start with an auxiliary lemma which will be used in later discussions. 

\begin{lem}\label{m_2}
Let $H(p):=(1/m)|p|^m$ with $m>1$, and set
\begin{equation}\label{h(p;q)}
h(p;q):=H(q+p)-H(q)-DH(q)\cdot p,\qquad p,q\in\R^d.
\end{equation}
Then, for any $r>0$ and $K>r$, there exists a $c>0$ such that $h(p; q)\geq c|p|^{2}$ for all $q\in \partial B_r:=\{w\in\R^d\,|\,|w|=r\}$ and $p\in B_K$.
\end{lem}

\begin{proof}
The proof is divided into two cases according to the value of $m$. We first consider the case where $1<m<2$. Let $q\in \partial B_r$ and $p\in B_K$. Suppose for a moment that $p_t:=q+(1-t)p\ne 0$ for all $t\in[0,1]$. Then, since $H\in C^2(\R^d\setminus \{0\})$, we see by Taylor's theorem that
\begin{align*}
h(p;q)&=H(q+p)-H(q)-DH(q)\cdot p\\
&=\int_0^1t\,D^2H(p_t)p\cdot p\,dt
=\int_0^1t\,|p_t|^{m-2}\left\{|p|^2+(m-2)\frac{(p_t\cdot p)^2}{|p_t|^2}\right\}dt\\
&\geq \int_0^1t\,|p_t|^{m-2}\left\{|p|^2+(m-2)\frac{|p_t|^2 |p|^2}{|p_t|^2}\right\}dt\\
&=(m-1)|p|^2\int_0^1t\,|p_t|^{m-2}dt,
\end{align*}
where $D^2H(p)$ denotes the Hessian matrix of $H(p)$.
Since $|p_t|\leq |q|+|p|< r+K<2K$, we have $|p_t|^{m-2}>(2K)^{m-2}$ for all $t\in[0,1]$. Thus,
\begin{equation*}
h(p;q)>
(m-1)|p|^2\int_0^1t\,(2K)^{m-2}dt=(m-1)2^{m-3}K^{m-2}|p|^2.
\end{equation*}
In view of the continuity of $h(p;q)$ in $p$, this inequality is still valid even if $p_t=0$ for some $t\in[0,1]$. Hence, our claim holds with $c:=(m-1)2^{m-3}K^{m-2}$. 
 
We next consider the case where $m\geq 2$. Then, for any $p\in B_K$,  
\begin{align*}
h(p;q)&=\int_0^1t\,|p_t|^{m-2}\left\{|p|^2+(m-2)\frac{(p_t\cdot p)^2}{|p_t|^2}\right\}dt
\geq |p|^2\int_0^1t\,|p_t|^{m-2}dt.
\end{align*}
Since $1-r/(2K)\leq t\leq 1$  implies that $|p_t|\geq |q|-(1-t)|p|> r-(r/2K)K=r/2$, we have
\begin{align*}
h(p;q)&> |p|^2\int_{1-r/(2K)}^1t\Big(\frac{r}{2}\Big)^{m-2}\,dt=\frac{(4K-r)r^{m-1}}{2^{m+1}K^2}|p|^2.
\end{align*}
Hence, our claim holds with $c:=(4K-r)r^{m-1}2^{-(m+1)}K^{-2}$.

In any case, $c$ does not depend on the choice of $q\in \partial B_r$ and $p\in B_K$. Hence, we have completed the proof. 
\end{proof}

Taking into account the previous lemma, one has the following result.

\begin{prop}\label{de=0_upper}
Let (H0) hold with $\de=0$. Assume that, for any $k>0$, there exist $\mu>0$ and $R_2\geq R$ such that
\begin{equation}\label{structure2}
|x|(\al(x)+k|\al(x)|^{2}+kV(x))\leq d-1-\mu,\qquad \forall x\not\in  B_{R_2}. 
\end{equation}
Then, $\lm_{\max}(\be)<\rho^{m^\ast}/m^\ast$ for all $\be\geq 0$. 
\end{prop}

\begin{proof}
We argue by contradiction assuming that $\lm_{\max}(\be)=\rho^{m^\ast}/m^\ast$ for some $\be>0$. In what follows, we fix such $\be$.
Let $(\rho^{m^\ast}/m^\ast,\phi)\in \R\times C^3(\R^d)$ be a solution of (EP). Then, 
\begin{equation}\label{eq_phi}
\frac{\rho^{m^\ast}}{m^\ast}-\De\phi(x)+b(x)\cdot D\phi(x)+\frac1m|D\phi(x)|^m=\be V(x), \quad x\in\R^d.
\end{equation}
Let $\phi_0\in C^3(\R^d)$ be any function such that $\phi_0(x)=-\rho^{m^\ast-1} |x|$ for $x\not\in B_R$. Set $V_0(x):=\rho^{m^\ast}/m^\ast+F[\phi_0](x)$ for $x\in\R^d$. Note that $V_0(x)=(d-1)\rho^{m^\ast-1}|x|^{-1}-\rho^{m^\ast-1} \al(x)$ for all $x\not\in B_R$. Then, we see that $\psi:=\phi-\phi_0$ satisfies
\begin{equation}\label{EP'}
-\De\psi(x)+b_0(x)\cdot D\psi(x)+H_0(x,D\psi(x))=\be V(x)-V_0(x),\quad x\in \R^d,
\end{equation}
with $b_0(x):=b(x)+|D\phi_0(x)|^{m-2}D\phi_0(x)$ and $H_0(x,p)=h(p;D\phi_0(x))$, where $h(p;q)$ is given by (\ref{h(p;q)}). Applying Lemma \ref{m_2} with $r=\rho^{m^\ast-1}$, $q=D\phi_0(x)$, and $K:=\max\{\sup_{\R^d}|D\psi|,r+1\}$, we conclude that 
\begin{equation*}
H_0(x,D\psi(x))\geq c|D\psi(x)|^{2},\qquad x\not\in B_R,
\end{equation*}
for some $c>0$. 
Since $b_0(x)=\al(x)|x|^{-1}x$ for $x\not\in B_R$, we see from (\ref{EP'}) that
\begin{align*}
-\De\psi(x)&+ (\al(x)|x|^{-1}x)\cdot D\psi(x)+c|D\psi(x)|^{2}\\
&\leq \frac{\rho^{m^\ast-1}}{|x|}\left\{ |x|(\al(x)+\be\rho^{1-m^\ast} V(x))-(d-1)\right\},\quad x\not\in B_R.
\end{align*}
We apply the Cauchy-Schwarz inequality $(\al(x)|x|^{-1}x)\cdot D\psi\leq (c/2)|D\psi|^2+C|\al(x)|^2$ for some $C>0$ depending only on $c$ to obtain 
\begin{align*}
-\De\psi(x)&+\frac{c}{2}|D\psi(x)|^{2}\\
&\leq \frac{\rho^{m^\ast-1}}{|x|}\left\{ |x|(\al(x)+C\rho^{1-m^\ast}|\al(x)|^{2}+\be\rho^{1-m^\ast} V(x))-(d-1)\right\}
\end{align*}
for $x\not\in B_R$. 
We then use (\ref{structure2}) with $k:=\rho^{1-m^\ast}\max\{C,\be\}$ to deduce that 
\begin{equation*}
-\De\psi+\frac{c}{2}|D\psi|^{2}\leq -\frac{\rho^{m^\ast-1}\mu}{|x|},\quad x\not\in B_{R_2},
\end{equation*}
for some $\mu>0$ and $R_2>0$.  

Now, fix a test function $\zeta\in C^\infty_0(\R^d)$ such that $\zeta\geq 0$ in $\R^d$, $\int_{\R^d} \zeta^{2}\,dx=1$, and $(\supp\zeta)\cap B_{R_2}=\emptyset$. Then, multiplying both sides of the previous estimate by $\zeta^{2}$ and using the integration by parts formula, we obtain
\begin{align*}
2\int_{\R^d}\zeta D\zeta\cdot D\psi\,dx +\frac{c}{2}\int_{\R^d}\zeta^{2} |D\psi|^{2}\,dx
 \leq -\rho^{m^\ast-1}\mu\int_{\R^d}\frac{\zeta(x)^2}{|x|}\,dx. 
\end{align*}
We apply Young's inequality to $D\zeta\cdot (\zeta D\psi)$ to obtain
\begin{equation*}
\int_{\R^d}\frac{\zeta(x)^{2}}{|x|}\,dx
\leq C\int_{\R^d}|D\zeta(x)|^{2}\,dx
\end{equation*}
for some $C>0$ not depending on $\zeta$. Plugging $\zeta_\ep(x):=\ep^{\frac{d}{2}}\zeta(\ep x)$, with $\ep\in(0,1)$, into the above $\zeta$, and noting that $\supp(\zeta_\ep)\subset (B_{\ep^{-1}R_2})^c\subset (B_{R_2})^c$, we have 
\begin{align*}
&\int_{\R^d}\frac{\zeta_\ep(x)^2}{|x|}\,dx
=\int_{\R^d}\frac{\ep^d\zeta(\ep x)^2}{|x|}\,dx
=\ep\int_{\R^d}\frac{\zeta(y)^2}{|y|}\,dy,\\
&\int_{\R^d}|D\zeta_\ep(x)|^2\,dx
=\int_{\R^d}\ep^2\ep^d|D\zeta(\ep x)|^2\,dx
=\ep^2\int_{\R^d}|D\zeta(y)|^2\,dy.
\end{align*}
Gathering these, we obtain
\begin{equation*}
\ep\int_{D}\frac{\zeta(y)^2}{|y|}\,dy
\leq \ep^2C\int_{D}|D\zeta(y)|^2\,dy,
\end{equation*}
which deduces a contradiction by dividing both sides of the last inequality by $\ep$ and sending $\ep\to0$. 
Hence, $\lm_{\max}(\be)<\rho^{m^\ast}/m^\ast$ for all $\be\geq 0$.
\end{proof}

\begin{prop}\label{de=0_bar}
Under the hypothesis of Proposition \ref{de=0_upper}, $\overline{\lm}_{\max}=\rho^{m^\ast}/m^\ast$.
\end{prop}

\begin{proof}
Let $\phi_0$ be as in the proof of Proposition \ref{de=0_upper}. Fix any $\ep>0$. Then, for any $x\not\in B_R$,
\begin{equation*}
\frac{\rho^{m^\ast}}{m^\ast}-\ep+F[\phi_0](x)=(d-1)\rho^{m^\ast-1}|x|^{-1}-\rho^{m^\ast-1}\al(x)-\ep.
\end{equation*}
Since $\al(x)\to0$ as $|x|\to\infty$, there exists an $R'\geq R$ such that, for any $\be\geq 0$,
\begin{equation*}
\frac{\rho^{m^\ast}}{m^\ast}-\ep+F[\phi_0](x)-\be V(x)\leq 0,\qquad x\not\in B_{R'}.
\end{equation*}
Choosing $\be$ so large that
$\be\geq \be_0:=(\inf_{B_{R'}}V)^{-1}(\rho^{m^\ast}/m^\ast+\sup_{B_{R'}}F[\phi_0])$,
we have
\begin{equation*}
\frac{\rho^{m^\ast}}{m^\ast}-\ep+F[\phi_0](x)-\be V(x)\leq 0,\quad x\in B_{R'}.
\end{equation*}
Thus, $(\rho^{m^\ast}/m^\ast-\ep,\phi_0)$ is a subsolution of (EP), which implies that $\lm_{\max}(\be)\geq\rho^{m^\ast}/m^\ast-\ep$ by the definition of $\lm_{\max}(\be)$. Since $\ep>0$ is arbitrary, we obtain the claim in view of Proposition \ref{de=0}.
\end{proof}

We are in position to discuss the asymptotic behavior of $\lm_{\max}(\be)$ as
$\be\to\infty$ under (H1). Recall that, in view of Theorem \ref{pre2},
$\lm_{\max}(\be)$ converges as $\be\to\infty$ to
$\overline{\lm}_{\max}:=\sup\{\lm_{\max}(\be)\,|\,\be\geq 0\}<\infty$, and that
one of the following (a) and (b) occurs:
\vskip2mm
(a)  there exists a $\be_0>0$ such that $\lm_{\max}(\be)=\overline{\lm}_{\max}$ for all $\be\geq \be_0$;

(b)  $\lm_{\max}(\be)<\overline{\lm}_{\max}$ for every $\be\geq 0$.  
\vskip2mm
\noindent The next theorem gives a characterization of the above dichotomy in terms of constants $a,\eta$ in (H1). 

\begin{thm}\label{main2}
Let (H1) hold with $\de=0$. Then $\overline{\lm}_{\max}=\rho^{m^\ast}/m^\ast$. Moreover, the following hold:\\
 (i) \ If $0<\eta\leq 1$ or $a\geq d-1$, then (a) occurs.\\
(ii) \ If $\eta> 1$ and $a< d-1$, then (b) occurs.
\end{thm}

\begin{proof}
Suppose first that $0<\eta\leq 1$. Then, $|x|(\al(x)+k_1V(x))=a+k_1|x|^{1-\eta}$ for all $x\not\in B_{R}$. Choosing $k_1>0$ and $R_1>R$ so large that $a+k_1R_1^{1-\eta}\geq d-1$, we see that (\ref{structure1}) holds. Thus, in virtue of Propositions \ref{de=0} and \ref{de=0_lower}, we conclude that (a) occurs with $\overline{\lm}_{\max}=\rho^{m^\ast}/m^\ast$. 
Suppose next that $a\geq d-1$. Then, for any $k_1>0$ and $x\not\in B_{R}$, we have $|x|(\al(x)+k_1V(x))\geq a\geq  d-1$, which implies that (\ref{structure1}) holds. Hence, (a) occurs with $\overline{\lm}_{\max}=\rho^{m^\ast}/m^\ast$.

Finally, we suppose that $\eta> 1$ and $a< d-1$. We choose a $\mu>0$ so small that $2\mu<d-1-a$. Then, for any  $k>0$ and $x\not\in B_{R}$,
\begin{align*}
|x|(\al(x)+k|\al(x)|^2+kV(x))&=a+k|a|^2|x|^{-1}+k|x|^{1-\eta}\\
&<d-1-2\mu+k|a|^2|x|^{-1}+k|x|^{1-\eta}.
\end{align*}
Since $\eta>1$, one can find an $R_2>R$ such that $k|a|^2R_2^{-1}+kR_2^{1-\eta}<\mu$. Thus, 
\begin{align*}
|x|(\al(x)+k|\al(x)|^2+kV(x))<d-1-\mu,\quad x\not\in B_{R_2},
\end{align*}
which implies that (\ref{structure2}) holds. In view of Propositions \ref{de=0}, \ref{de=0_upper}, and \ref{de=0_bar}, we conclude that  (b) occurs with $\overline{\lm}_{\max}=\rho^{m^\ast}/m^\ast$. 
\end{proof}

\begin{rem}
Using (\ref{structure1}) or  (\ref{structure2}), one can derive various conditions on $\al$ and $V$ that guarantee (a) or (b). For instance, suppose that
\begin{equation*}
a_1|x|^{-1}\leq \al(x)\leq  a_2|x|^{-1},\qquad 
v_1|x|^{-\eta}\leq V(x)\leq  v_2|x|^{-\eta},\qquad x\not\in B_R,
\end{equation*}
for some $a_1,a_2\in\R$, $v_1,v_2>0$,  $\eta>0$, and $R>0$. Then, we are able to
specify suitable sufficient conditions, in terms of the constants above, so that
(a) or (b) hold. 
\end{rem}

By a careful reading of the arguments used in this section, one observes that the positivity of $V$ does not play any role outside $B_R$, although it is crucial in $B_R$. Taking this into account, one can prove the following theorem.  
\begin{thm}\label{main2+}
Let (H0) hold with $\de=0$. Assume that $\al(x)=a|x|^{-1}$ outside $B_R$ for some $a\in\R$ and $\supp V\subset B_{R'}$ for some $R'> R$, where $R$ is the constant in (H0). Then, $\overline{\lm}_{\max}=\rho^{m^\ast}/m^\ast$. Moreover, the following hold.\\
(i) \ If $a\geq d-1$, then (a) occurs. \\
(ii) \ If $a<d-1$, then (b) occurs.
\end{thm}

\section{The sharp estimate for $\de=0$}
In this section, we establish a sharp estimate of the form (\ref{c_1}) when (b)
occurs in the moderate drift case. 

\begin{thm}\label{main3}
Let (H1) hold with $\de=0$. Assume that $a$ and $\eta$ satisfy condition (ii) in Theorem \ref{main2}. Then, 
\begin{equation}\label{sharp2}
\lim_{\be\to \infty}\be^{\frac{1}{\eta-1}}\left(\frac{\rho^{m^\ast}}{m^\ast}-\lm_{\max}(\be)\right)
=(\eta-1)\left\{\frac{\rho^{m^\ast-1}}{\eta}(d-1-a)\right\}^{\frac{\eta}{\eta-1}}=:c_1.
\end{equation}
\end{thm}

As in the proof of Theorem \ref{main1}, we divide (\ref{sharp2}) into upper and lower estimates.  

\begin{prop}
Under the assumption of Theorem \ref{main3}, 
\begin{equation}\label{upper3}
\limsup_{\be\to \infty}\be^{\frac{1}{\eta-1}}\left(\frac{\rho^{m^\ast}}{m^\ast}-\lm_{\max}(\be)\right)
\leq c_1.
\end{equation}
\end{prop}

\begin{proof}
Let $\phi_0\in C^3(\R^d)$ be any function such that $\phi_0(x)=-\rho^{m^\ast-1}|x|$ for all $x\not\in B_R$. Then, by direct computations, we see that, 
\begin{align*}
F[\phi_0](x)-\be V(x)=-\frac{\rho^{m^\ast}}{m^\ast}+\rho^{m^\ast-1}(d-1-a)|x|^{-1}-\be|x|^{-\eta},\quad x\not\in B_R.
\end{align*}
We set $\kp:=\rho^{m^\ast-1}(d-1-a)$ and $f(r):=\kp r^{-1}-\be r^{-\eta}$ for $r\geq R$. Since $\kp>0$ and $\eta>1$, we observe by direct computations that $f$ attains its maximum at $r=\max\{R,r_0\}$, where $r_0:=(\eta\be/\kp)^{\frac{1}{\eta-1}}$. By taking $\be$ so large that $(\eta\be/\kp)^{\frac{1}{\eta-1}}> R$, we obtain
\begin{equation*}
\frac{\rho^{m^\ast}}{m^\ast }-f(r_0)+F[\phi_0](x)-\be V(x)\leq 0,\qquad x\not\in B_R.
\end{equation*}
Furthermore, replacing $\be$ by a larger one which satisfies
\begin{equation*}
\be>\frac{\rho^{m^\ast}/m^\ast -f(r_0)+\sup_{B_R}F[\phi_0]}{\inf_{B_R}V},
\end{equation*}
we also have 
\begin{equation*}
\frac{\rho^{m^\ast}}{m^\ast }-f(r_0)+F[\phi_0](x)-\be V(x)\leq 0,\qquad x\in B_R.
\end{equation*}
Thus, $(\rho^{m^\ast}/m^\ast -f(r_0),\phi_0)$ is a subsolution of (EP). This implies by the definition of $\lm_{\max}(\be)$ that 
\begin{align*}
\frac{\rho^{m^\ast}}{m^\ast}-\lm_{\max}(\be)
\leq  f(r_0)
&=\kp\Big(\frac{\eta\be}{\kp}\Big)^{-\frac{1}{\eta-1}}-\be \Big(\frac{\eta\be}{\kp}\Big)^{-\frac{\eta}{\eta-1}}\\
&=(\eta-1)\Big(\frac{\kp}{\eta}\Big)^{\frac{\eta}{\eta-1}}\be^{-\frac{1}{\eta-1}}=c_1\be^{-\frac{1}{\eta-1}}
\end{align*}
for any $\be$ sufficiently large. Hence, we obtain (\ref{upper3})
\end{proof}

We turn to the proof of the lower bound.
\begin{prop}
Under the assumption of Theorem \ref{main3}, 
\begin{equation}\label{lower3}
\liminf_{\be\to \infty}\be^{\frac{1}{\eta-1}}\left(\frac{\rho^{m^\ast}}{m^\ast}-\lm_{\max}(\be)\right)
\geq c_1.
\end{equation}
\end{prop}

\begin{proof}
Let $(\lm_{\max}(\be),\phi)\in\R\times C^3(\R^d)$ be a solution to (EP), and let $\phi_0\in C^3(\R^d)$ be any function such that $\phi_0(x)=-\rho^{m^\ast-1}|x|$ for all $x\not\in B_R$. Set $\psi:=\phi-\phi_0$. Then, similarly as in the proof of Proposition \ref{de=0_upper}, there exists some $c>0$ such that, for any $x\not\in B_R$, 
\begin{align*}
\frac{\rho^{m^\ast}}{m^\ast}-\lm_{\max}(\be)
&\geq -\De\psi+(a|x|^{-2}x)\cdot D\psi(x)+c|D\psi(x)|^2\\
&\qquad\qquad\qquad +\rho^{m^\ast-1}(d-1-a)|x|^{-1}-\be |x|^{-\eta}.
\end{align*}
Applying the Cauchy-Schwarz inequality to $(a|x|^{-2}x)\cdot D\psi$, we obtain
\begin{align*}
\frac{\rho^{m^\ast}}{m^\ast}-\lm_{\max}(\be)
&\geq -\De\psi(x)-C|x|^{-2}+\frac{c}{2}|D\psi(x)|^{2}+\kp|x|^{-1}-\be |x|^{-\eta}
\end{align*}
for all $x\not\in B_R$, where $\kp:=\rho^{m^\ast-1}(d-1-a)$ and $C>0$ is a constant depending only on $c$ and $a$. 

We now fix any $\zeta\in C_0^\infty(0,\infty)$ such that $\zeta\geq 0$ in $(0,\infty)$, $\supp\zeta\subset (R,R+1)$, and  $\int_{\R^d} \zeta(|x|)^{2}\,dx=1$. Furthermore, we define $\zeta_\ep\in C^\infty_0(\R^d)$, with $0<\ep<1$, by $\zeta_\ep(x):=\ep^{\frac{d}{2}}\zeta(\ep |x|)$. We also set $f(r):=\kp r^{-1}-\be r^{-\eta}$ for $r\geq R$. Then, since $\supp \zeta_\ep\subset (B_{\ep^{-1}R})^c\subset (B_{R})^c$, we see by the integration by parts formula and the Cauchy-Schwarz inequality that
\begin{align*}
\frac{\rho^{m^\ast}}{m^\ast}-\lm_{\max}(\be)
&\geq 2 \int_{\R^d}D\zeta_\ep\cdot (\zeta_\ep D\psi)\,dx -C\int_{\R^d}\zeta_\ep^2|x|^{-2}\,dx\\
&\qquad\qquad +\frac{c}{2}\int_{\R^d}\zeta_\ep^{2} |D\psi|^2\,dx+\int_{\R^d}\zeta_\ep^2f(|x|)\,dx\\ 
& \geq -C\int_{\R^d}|D\zeta_\ep(x)|^2dx -C\int_{\R^d}\zeta_\ep^2|x|^{-2}\,dx+\int_{\R^d}\zeta_\ep^2f(|x|)\,dx.
\end{align*}
Here and in the following, $C>0$ denotes various constants not depending on $\ep$ and $\be$. Hence, we obtain 
\begin{align*}
\frac{\rho^{m^\ast}}{m^\ast}-\lm_{\max}(\be)
\geq  -C\ep^2\int_{\R^d}\left(|D\zeta(y)|^2 +\frac{\zeta(y)^2}{|y|^2}\right)dy+\int_{\R^d}\zeta(y)^2 f\Big(\frac{|y|}{\ep}\Big)\,dy.
\end{align*}
We fix an arbitrary $\theta>0$ and choose $\ep$ so that $\ep\be^{\frac{1}{\eta-1}}=\theta$. Then, 
\begin{align*}
\be^{\frac{1}{\eta-1}}\Big(\frac{\rho^{m^\ast}}{m^\ast}-\lm_{\max}(\be)\Big)
& \geq  -C\be^{-\frac{1}{\eta-1}}\theta^2\int_{\R^d}\left(|D\zeta(y)|^2 +\frac{\zeta(y)^2}{|y|^2}\right)dy\\
&\qquad +\int_{\R^d}\zeta(y)^2\be^{\frac{1}{\eta-1}}f(\be^{\frac{1}{\eta-1}}\theta^{-1}|y|)\,dy.
\end{align*}
Since $
\be^{\frac{1}{\eta-1}}f(\be^{\frac{1}{\eta-1}}\theta^{-1}r)=\kp (\theta^{-1} r)^{-1}-(\theta^{-1} r)^{-\eta}=:g_\theta(r)$
for any $\be$, $\theta$, and $r$, we have
\begin{align}\label{g_theta}
\liminf_{\be\to\infty}\be^{\frac{1}{\eta-1}}\Big(\frac{\rho^{m^\ast}}{m^\ast}-\lm_{\max}(\be)\Big)
\geq \int_{\R^d}\zeta(y)^2g_\theta(|y|)\,dy.
\end{align}
Observe here that $g_\theta(r)$ attains its maximum at $r=\max\{R,r_\theta\}$, where $r_\theta:=\theta(\eta/\kp)^{\frac{1}{\eta-1}}$, and that $g_\theta(r_\theta)=(\eta-1)(\kp /\eta)^{\frac{\eta}{\eta-1}}=c_1$. Thus, letting $\theta$ so that $R<r_\theta<R+1$ and choosing $\zeta$ so that the right-hand side of (\ref{g_theta}) is arbitrarily close to $g_\theta(r_\theta)=c_1$, we obtain the desired estimate.
\end{proof}

\subsection*{Acknowledgment}
EC is partially supported by ANR-16-CE40-0015-01 (ANR project on Mean Field Games).
NI is supported in part by JSPS KAKENHI Grant Number 18K03343.

\end{document}